\documentclass{amsart}
\usepackage{Preamble}
\newcommand\blfootnote[1]{
    \begingroup
    \renewcommand\thefootnote{}\footnote{#1}
    \addtocounter{footnote}{-1}
    \endgroup
}

\title{Equivariant Double-Slice Genus, Stabilization, and Equivariant Stabilization}
\author{Malcolm Gabbard}
\begin{document}

\begin{abstract} In this paper we define the equivariant double-slice genus and equivariant super-slice genus of a strongly invertible knot. We prove lower bounds for both the equivariant double-slice genus and the equivariant super-slice genus. Using these bounds we find a family of knots which are double-slice and equivariantly slice, but have equivariant double-slice genus at least $n$. Using this result, we construct unknotted symmetric 2-spheres which do not bound symmetric 3-balls. Additionally, using double-slice and super-slice genera we find effective lower bounds for 1-handle stabilization distance and identify a possible method for using equivariant double-slice and super-slice genera to bound symmetric 1-handle stabilization distance for symmetric surfaces.
    
\end{abstract}

\maketitle

\section{Introduction}

Given a knot $K\subset S^3$, its \textit{double-slice genus} $g_{ds}(K)$ was first defined by Livingston-Meier in \cite{livingston2015doubly} as the minimal genus of an unknotted surface $\Sigma \subset S^4$ such that $\Sigma$ intersects an equatorial $S^3$ transversely with intersection $K$. In 2021, Chen \cite{chen2021lower} was able to construct a lower bound for the double-slice genus using Casson-Gordon invariants which allowed him to prove that there are slice knots with arbitrarily large double-slice genus. A later result by Orson-Powell in \cite{orson2020lower} using a new lower bound for double-slice genus from signatures refined this result finding slice knots with double-slice genus exactly $n$ for all $n$.

The discussion of equivariant genus of symmetric knots originates from Sakuma's work \cite{sakuma1986strongly} on the equivariant concordance group of strongly invertible knots. A \textit{strongly invertible knot} $(K,\tau)$ is a knot $K\subset S^3$ and an involution $\funct{\tau}{S^3}{S^3}$ satisfying $\tau(K)=~K$ and $fix(\tau)=S^1$ intersecting $K$ in two points. The equivariant concordance group of strongly invertible knots has received considerable attention in recent years including \cite{boyle2022equivariant, cha1999equivariant, di2023equivariant, di2023new}, providing many obstructions to strongly invertible knots being equivariantly slice, i.e. bounding a symmetric disk in $B^4$. A more general analysis was recently initiated by Boyle-Issa in \cite{boyle2022equivariant} where they define for a strongly invertible knot $(K,\tau)$ its \textit{equivariant 4-genus} $\tilde{g}_4(K,\tau)$ to be the minimal genus of a properly embedded surface $\Sigma \subset B^4$ with $\partial \Sigma =K$ such that for some extension $\funct{\overline{\tau}}{B^4}{B^4}$ of $\tau$, $\overline{\tau}(\Sigma)=\Sigma$. Using knot Floer homology, Dai-Mallick-Stoffregen \cite{dai2023equivariant} were able to find slice knots which, viewed as strongly invertible knots, have arbitrarily large equivariant 4-genus.\blfootnote{The author was partially supported by NSF grant DMS-1952755.}

Combining these ideas, we define the \textit{equivariant double-slice genus} $\tilde{g}_{ds}(K,\tau)$ of a strongly invertible knot $(K,\tau)$ to be the minimal genus of an equivariantly unknotted surface $\Sigma\subset S^4$ of which $K$ appears as the cross-section. In other words, the minimal genus of a surface $\Sigma$ intersecting $S^3$ transversly with intersection $K$ and bounding a handlebody $H$ such that $\overline{\tau}(H)=H$ for some orientation preserving extension of $\tau$ to $S^4$. In order to differentiate the equivariant double-slice genus from the equivariant 4-genus and the double-slice genus, we prove the following lower bound:

\begin{theorem} \label{thm: Main} Let $(K,\tau)$ be a strongly invertible knot and let $K_0$ and $K_1$ be the knots formed from an arc of $K$ union the half-axis $h_0$ and $h_1$ respectively. Then:
$$\textnormal{min}\{g_{ds}(K_0),g_{ds}(K_1)\}\,\leq \tilde{g}_{ds}(K,\tau).$$
\end{theorem}

This lower bound allows, to some extent, questions about the equivariant double-slice genus of a strongly invertible knot to be answered in terms of the non-equivariant double-slice genus of these other knots $K_0$ and $K_1$, allowing us to make use of the well-studied bounds for double-slice knots. Using this lower bound, we are able to prove that the equivariant double-slice genus of a knot $(K,\tau)$ does not depend on a combination of its double-slice genus and equivariant 4-genus. 
\begin{theorem}\label{thm: Construction}
    The knot $(K_n,\tau)$ depicted in Figure \ref{fig:Main Construction} satisfies the following:
    \begin{enumerate}
        \item $K_n$ is double-slice,
        \item $(K_n,\tau)$ is equivariantly slice,
        \item $\tilde{g}_{ds}(K_n,\tau)\geq n$.
    \end{enumerate}
\end{theorem}

\begin{figure}[ht]
    \centering
    \includegraphics[width=.4 \linewidth]{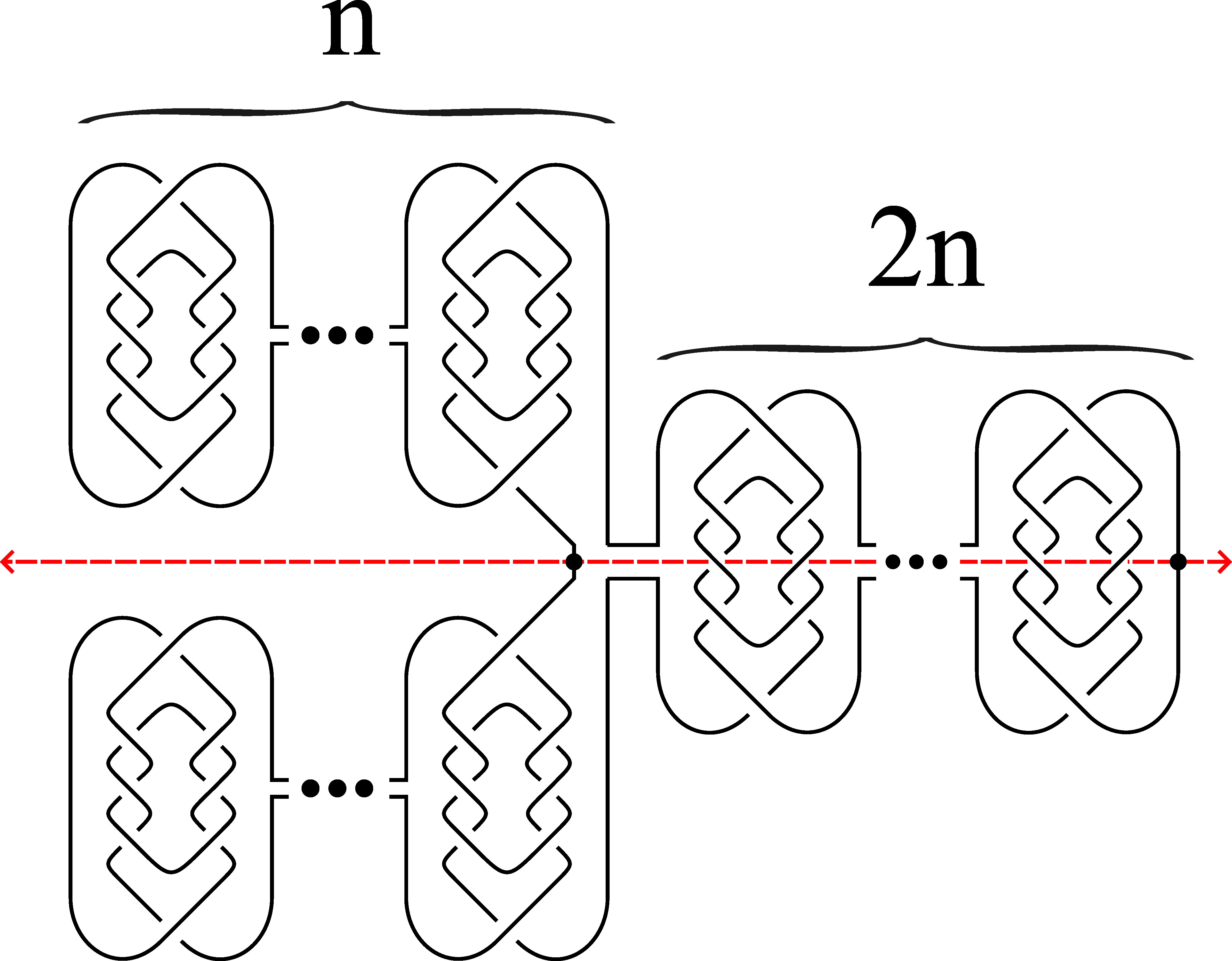}
    \caption{$(K_n,\tau)$}
    \label{fig:Main Construction}
\end{figure}

\subsection{Equivariantly super-slice knots}

Similar to double-slice genus, there is a notion of \textit{super-slice} genus $g_{ss}(K)$ defined by Chen in \cite{chen2021lower} as the minimal genus of a surface $\Sigma \subset B^4$ with boundary $K$ such that the double of $\Sigma$ along $K$ in $S^4$ is unknotted. We extend this discussion to the equivariant setting by defining the \textit{equivariant super-slice genus} $\tilde{g}_{ss}(K,\tau)$ of a knot to be the minimal genus of a surface $\Sigma\subset B^4$ such that $\overline{\tau}(\Sigma)=\Sigma$ and the double of $\Sigma$ is equivariantly unknotted. Using similar techniques to the double-slice setting we are able to prove the following:

\begin{theorem} \label{thm: super-slice Intro} Let $(K,\tau)$ be a strongly invertible knot and let $K_0$ and $K_1$ be the knots formed from an arc of $K$ union the half-axis $h_0$ and $h_1$ respectively. Then:
$$\textnormal{min}\{g_{ss}(K_0),g_{ss}(K_1)\}\,\leq \tilde{g}_{ss}(K,\tau).$$
\end{theorem}

\subsection{Symmetric 2-Knots}

Similar to strongly invertible knots, one could ask what are the properties of 2-knots invariant under some symmetry of $S^4$. Using Theorem \ref{thm: super-slice Intro}, we are able construct interesting examples of symmetric 2-knots that distinguish them significantly from both non-equivariant 2-knots and strongly invertible knots. Namely, we are able to prove the following result:

\begin{theorem}\label{thm: 2-sphere}
    There exists a symmetric 2-sphere in $(S^4,\overline{\tau})$ which bounds a 3-ball but bounds no equivariant 3-ball.
\end{theorem}

\subsection{Stabilization distance of disks rel boundary}

The internal stabilization distance of embedded surfaces in 4-manifolds is a well-studied area \cite{auckly2019isotopy,Baykur_2015, guth2022exotic,hayden2023atomic, hayden2023one, juhasz2018stabilization,miller2020stabilization, singh2020distances},  which generally asks the following: given two embedded surfaces $\Sigma_1$ and $\Sigma_2$, how many internal stabilizations (additions of 2-dimensional 1-handles to $\Sigma_1$ and $\Sigma_2$) are necessary before $\Sigma_1$ and $\Sigma_2$ are isotopic? In recent work of Miller and Powell \cite{miller2020stabilization}, this question was extended to internal stabilizations of properly embedded surfaces rel boundary. In particular, they define the \textit{1-handle stabilization distance} $d_1(\Sigma_1,\Sigma_2)$ between smoothly and properly embedded genus $g$ surfaces $\Sigma_1,\Sigma_2\subset B^4$ with common boundary $K$ to be the minimal $n\subset \N$ such that $\Sigma_1$ and $\Sigma_2$ become ambiently isotopic rel boundary after each has been stabilized at most $n$ times. 

Using basic properties of double-slice and super-slice genus, we are able to prove the following lower bound for the 1-handle stabilization distance of surfaces satisfying certain conditions:

\begin{theorem}\label{thm: stab}
    Let $\Sigma_1$ and $\Sigma_2$ be properly embedded genus $h$ surfaces with boundary $K\subset S^3$ such that $\Sigma_1\cup_K \Sigma_2 \subset (B^4,\Sigma_1)\cup_{(S^3, K)}(B^4,\Sigma_2)$ is unknotted. Then $d_1(\Sigma_1,\Sigma_2)\geq g_{ss}(K)-h$.
\end{theorem}

By letting $K$ be double-slice this immediately yields the following corollary:

\begin{corollary}\label{cor: stab}
    Let $K$ be double-slice with $g_{ss}(K)=n$, then $K$ admits slice disks $D_1$ and $D_2$ such that $d_1(D_1,D_2)\geq n$.
\end{corollary}

From here, we define a symmetric notion of 1-handle stabilization distance $\tilde{d}^\tau_1(\Sigma_1,\Sigma_2)$ for certain classes of symmetric surfaces and prove the following symmetric analog of Theorem \ref{thm: stab}.

\begin{theorem}\label{thm: eq stabilization}
    Let $\Sigma_1,\Sigma_2\subset B^4$ be properly embedded genus $h$ surfaces with boundary $K$ which are both $\overline{\tau}$-invariant. If $\Sigma_1\cup_K -\Sigma_2 \subset (B^4,\Sigma_1)\cup_{(S^3,K)}(B^4,\Sigma_2)$ is equivariantly unknotted, then $\tilde{d}_1^\tau(\Sigma_1,\Sigma_2)\geq \frac{\tilde{g}_{ss}(K,\tau)}{2}-h$.
\end{theorem}

\subsection{Organization} Section 2 provides an overview of notation and conventions, as well as necessary background on double-slice genus and equivariant 4-genus. Section 3 defines the equivariant double-slice genus and covers some properties of involutions on handlebodies arising in this setting. In Section 4, we bound the equivariant double-slice genus, proving Theorem \ref{thm: Main} as well as Theorem \ref{thm: Construction}. In Section 5, we define equivariant super-slice genus and prove analogous results to Section 4, as well as results about equivariantly knotted 2-spheres.  In Section 6, we discuss stabilization of surfaces rel boundary and introduce a notion of equivariant stabilization.

\subsection{Acknowledgments} Many thanks to Dave Auckly for countless helpful conversations and Evan Scott for insights which helped find a mistake in an earlier version of the paper.

\section{Background}

In this section, we recall the necessary background pertaining to double-slice classical knots and equivariantly slice strongly invertible knots. 

\subsection{Conventions and classical 4-genus}

A knot $K$ will refer to a classical knot, namely the oriented image of a smooth embedding of $S^1$ into $S^3$. From a given knot $K$ we have the following related knots:

\begin{itemize}
    \item $rK$ is the \textit{reverse} of a knot $K$, that is, $K$ with the opposite orientation,
    \item $mK$ is the \textit{mirror} of a knot $K$, which is the image of $K$ under a reflection of $S^3$,
    \item $-K$ is the \textit{inverse} of a knot $K$, and is the reverse of the mirror.
\end{itemize}

Recall the \textit{4-genus} $g_4(K)$ of a knot $K$ is the minimal genus of smooth properly embedded surface $\Sigma \subset B^4$ with $\partial\Sigma = K$. If $g_4(K)=0$, we say $K$ is \textit{slice}. One well-known fact we will use consistently is that given a knot $K$, $K\# -K$ is slice. 

\subsection{Double-slice genus} 

The \textit{double-slice genus} of a knot $K$, denoted $g_{ds}(K)$, was defined by Livingston and Meier \cite{livingston2015doubly} as the minimal genus of an unknotted smooth surface $\Sigma \subset S^4$ such that the intersection of $\Sigma$ with an equatorial $S^3$ is $K$. By unknotted, we mean that the surface $\Sigma$ bounds a smoothly embedded handlebody in $S^4$. If the double-slice genus of $K$ is 0, i.e. $K$ is the cross-section of an unknotted sphere, we say $K$ is \textit{double-slice}. One important fact we know from Sumners \cite{sumners1971invertible} is that given any knot $K$, $K\# -K$ is double-slice. 

There are many obstructions to knots being double-slice, and a good survey is provided in \cite{livingston2015doubly}. In this paper, the primary bound for double-slice genus we will use is Theorem 1.1 of \cite{orson2020lower} from Orson and Powell, which states:

\begin{theorem}[Orson and Powell] Let $K$ be a knot in $S^3$ and $\sigma_\omega(K)$ be its signature function, then:
$$ g_{ds}(K) \geq \textnormal{max} | \sigma_\omega(K) |$$
\end{theorem}

One important application of this theorem is the following example coming from Theorem 1.2 of \cite{orson2020lower} which we will refer to in future constructions:

\begin{example}\label{example: 8_20}
    Let $L$ be the knot denoted $8_{20}$ in Rolfsen's table \cite{rolfsen2003knots} and let $L_n = \#^n L$ as depicted in Figure \ref{fig:8_20_sum}. In \cite{orson2020lower} it is shown, using the additivity of the knot signature function, that $g_{ds}(L_n)=n$, despite the fact that $L_n$ is slice.
\end{example}

\begin{figure}[h]
    \centering
    \includegraphics[width = .3\linewidth]{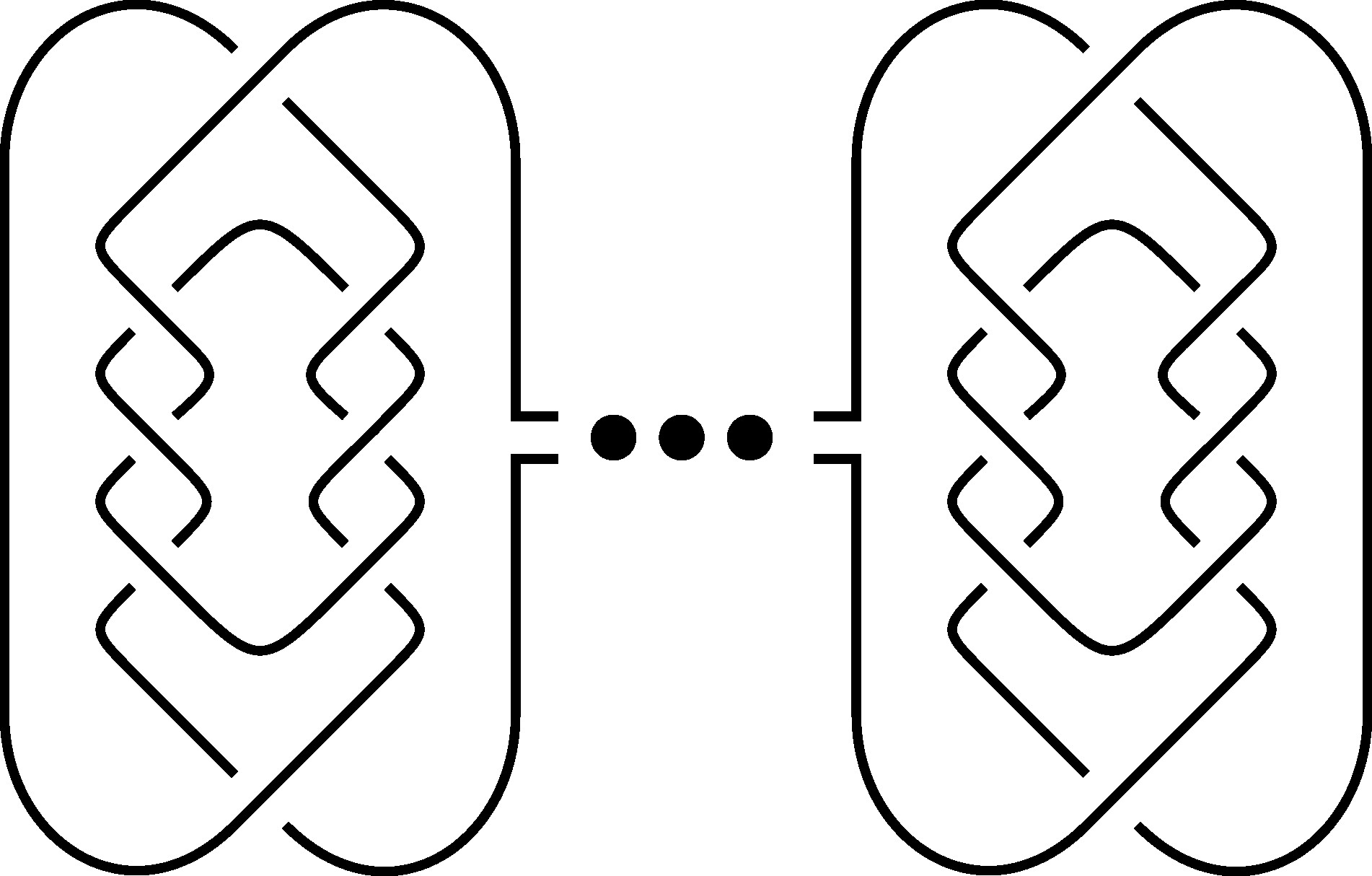}
    \caption{$L_n$, i.e. $n$ copies of $8_{20}$ summed together.}
    \label{fig:8_20_sum}
\end{figure}

\subsection{Equivariant 4-genus} \label{sec: Equivariant 4-Genus} To define the equivariant 4-genus, we first recall the definition of a strongly invertible knot. A \textit{strongly invertible knot} is a pair $(K,\tau)$ consisting of a classical knot $K$ and an involution $\tau$ acting on $S^3$ such that fix$(\tau)=~S^1$ and fix$(\tau)\cap K$ is two points. For a given strongly invertible knot $(K,\tau)$, we refer to the fixed point set as the \textit{axis} and refer to the two intervals of the axis separated by the intersection with $K$ as \textit{half-axis}. 

As it will often be useful to work with a specified half-axis, we recall that a \textit{directed strongly invertible knot} is a triple $(K,\tau,h)$ consisting of a strongly invertible knot $(K,\tau)$ and an oriented half-axis denoted by $h$. Given a directed strongly invertible knot $(K,\tau, h)$, its \textit{antipode} is the knot $a(K,\tau,h)=(K,\tau,h')$, where $h'$ is the other choice of half-axis; an example is shown in Figure \ref{fig: Antipode}.

\begin{figure}
\begin{tabular}{p{.35\textwidth}p{.35\textwidth}}
    \begin{center}\includegraphics[width=.6\linewidth]{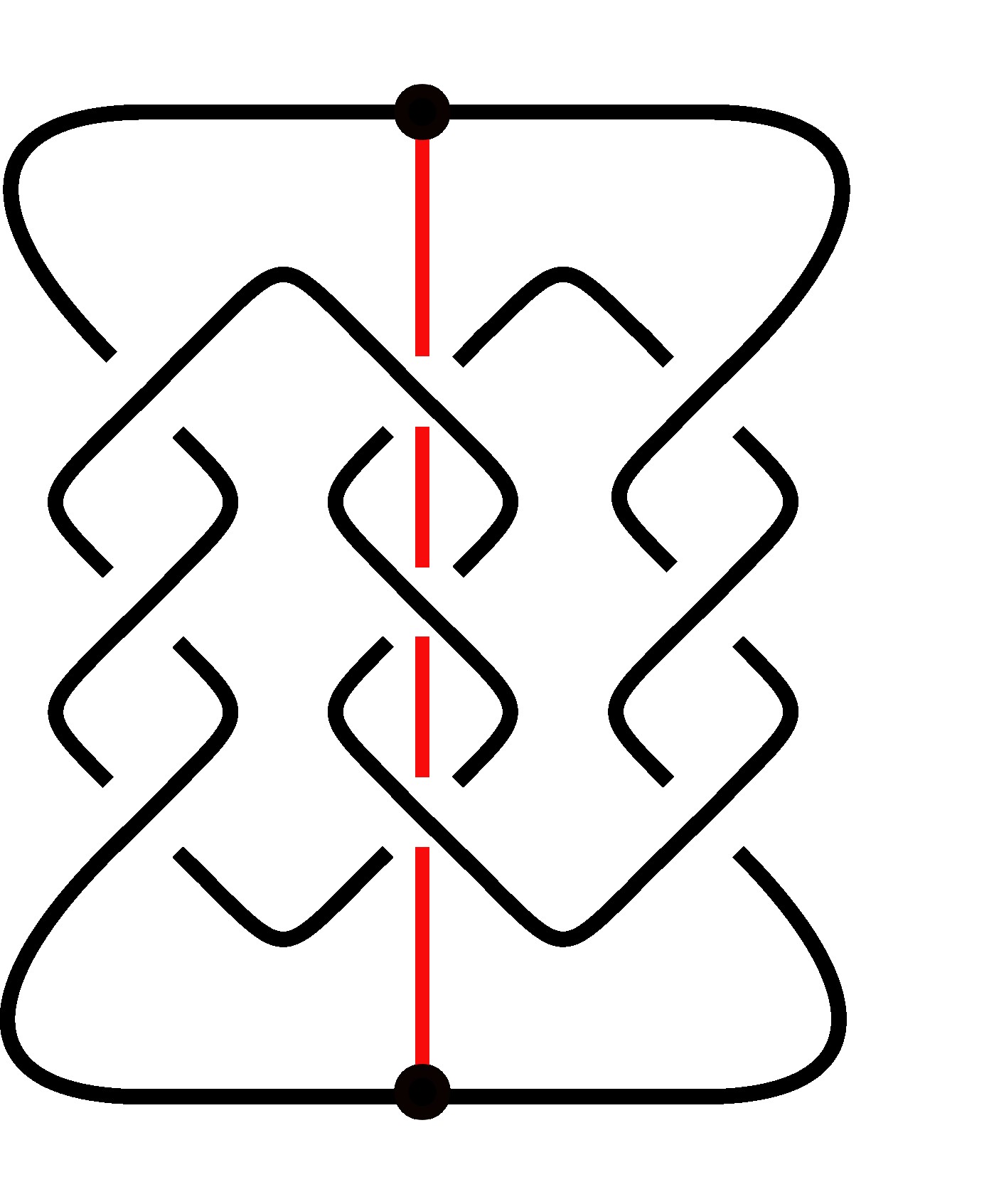}\end{center} &  \begin{center}\includegraphics[width=.6\linewidth]{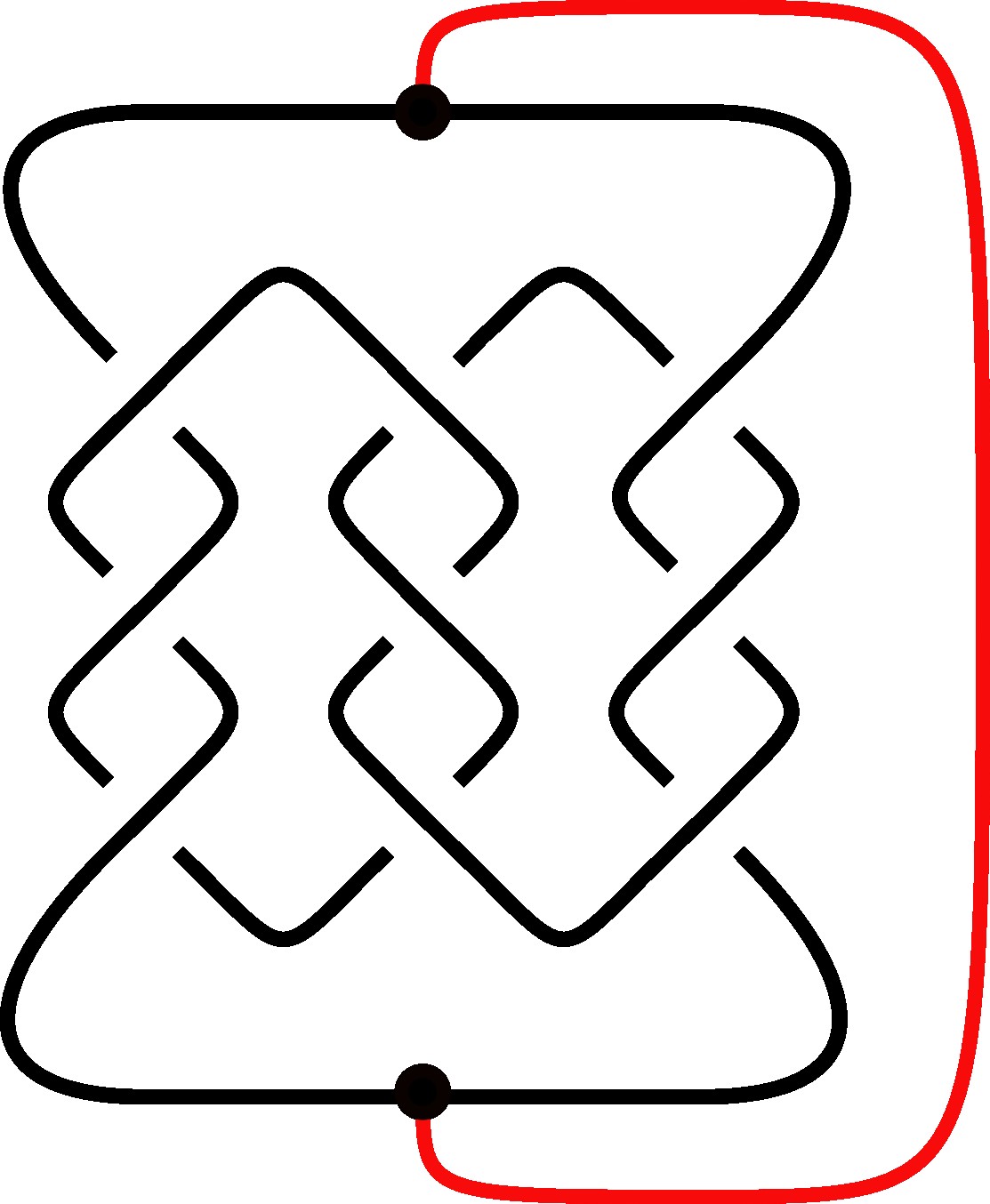} \end{center}
    \end{tabular}
\caption{A strongly invertible knot ($9_{46},\tau,h$) and its antipode.}
\label{fig: Antipode}
\end{figure}

Given a strongly invertible knot $(K,\tau)$, Boyle and Issa define in \cite{boyle2022equivariant} a notion of equivariant 4-genus $\tilde{g}_4(K,\tau)$, which we recall here. 

\begin{definition}
    Given a strongly invertible knot $(K,\tau)$ in $S^3$, an \textit{equivariant surface} for $(K,\tau)$ is a connected, smoothly properly embedded surface $F\subset B^4$ with $\partial F=K\subset \partial B^4$ such that $\overline{\tau}(F)=F$, for $\funct{\overline{\tau}}{B^4}{B^4}$ a smooth extension of $\tau$.
\end{definition}

\noindent From here we define the \textit{equivariant 4-genus} as in \cite{boyle2022equivariant}:

\begin{definition}
    The \textit{equivariant 4-genus} of a strongly invertible knot $(K,\tau)$ is the minimal genus of an orientable equivariant surface for $(K,\tau)$, denoted $\tilde{g}_4(K,\tau)$. When clear from context, we may instead write $\tilde{g}_4(K)$. 
\end{definition}

Given a strongly invertible knot $(K,\tau)$, we have that $(K\# -K,\tau)$ is equivariantly slice, where the connected sum is an equivariant connect sum banding together two directed strongly invertible knots with a symmetric band consistent with the orientations on the half-axis. For a more explicit description of the equivariant connect sum see \cite{sakuma1986strongly}. Another common operation we will use to construct strongly invertible knots is the \textit{equivariant double} of a knot $K$, denoted $D(K)$. 

\begin{figure}
    \centering
    \includegraphics[width = .45\linewidth]{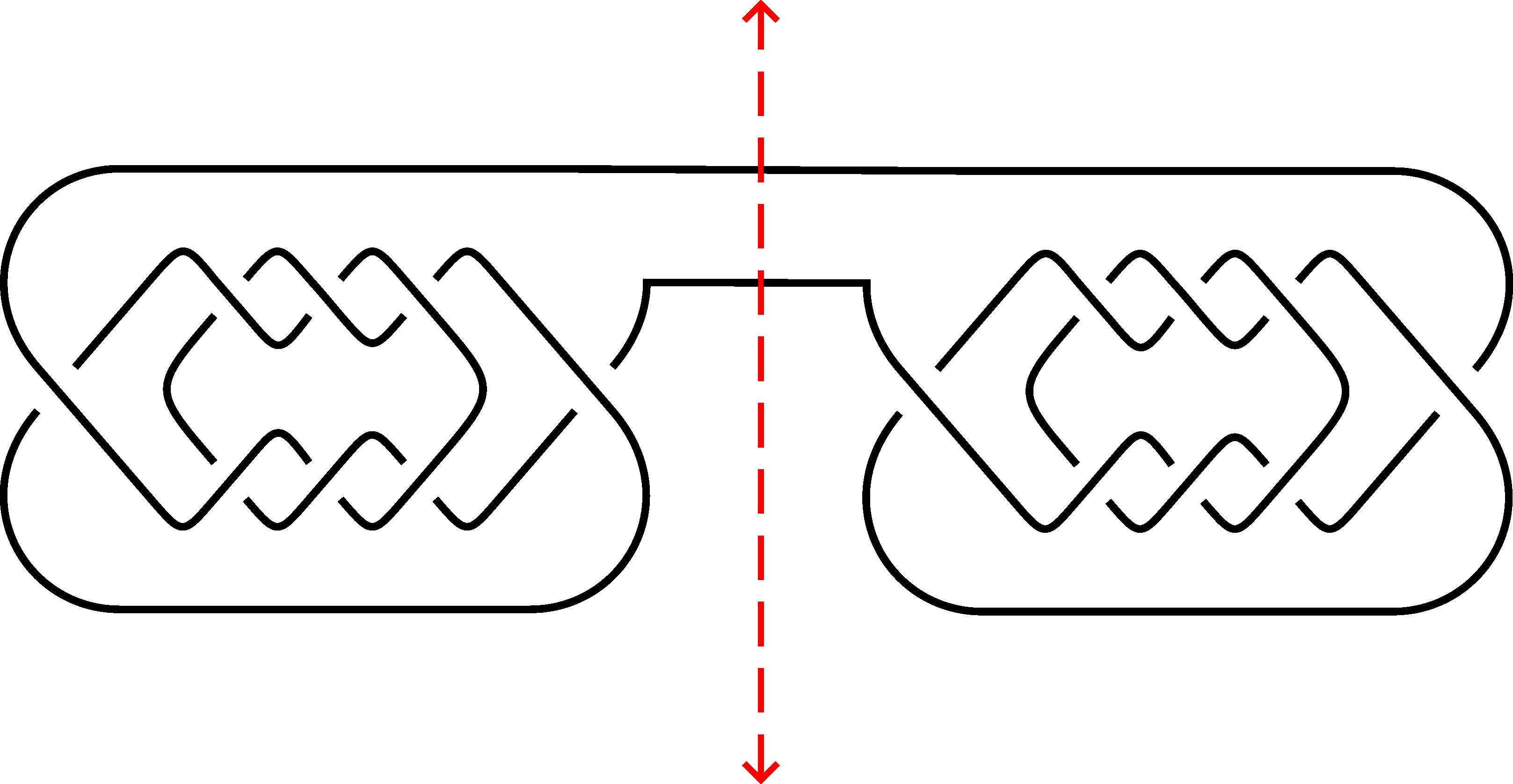}
    \caption{The equivariant double of $8_{20}$.}
    \label{fig:D820}
\end{figure}

The double of a knot $K$ is a strongly invertible knot $(D(K),\tau)$ where $D(K)=K\# rK$ and $\tau$ is the involution taking $K$ to $rK$ and vice-versa, as shown in Figure \ref{fig:D820}. Of importance to our genus bounds is the following immediate result:

\begin{proposition} \label{prop: double}
    Given a knot $K$ with 4-genus $g_4(K)=n$, its double $(D(K),\tau)$ has equivariant 4-genus $\tilde{g}_4(D(K))\leq 2n$.
\end{proposition}

\section{Equivariant double-slice genus}
\label{sec: EDSG}
Having defined both double-slice genus and equivariant 4-genus, we can now define the logical combination of the two, which we call the \textit{equivariant double-slice genus}. In order to do this, we first discuss extensions of $\tau$ to $S^4$. 

We divide the possible extension of $\tau$ into two categories: orientation preserving and orientation reversing. By Smith \cite{smith1941fixed} the fixed point set of these extensions will either be $S^2$ or $S^1$, depending on if it is orientation preserving or reversing, respectively. We will refer to an orientation preserving extension of $\tau$ to $S^4$ as $\overline{\tau}$, which will be the primary focus of this paper. We discuss orientation reversing extensions briefly at the end of this section.

With this in mind, we now define an \textit{equivariant slicing surface} and an \textit{equivariant slicing handlebody}: 

\begin{definition}
    Given a directed strongly invertible knot $(K,\tau,h)\subset S^3$ and a smooth extension $\overline{\tau}$ of $\tau$ to $S^4$, an \textit{equivariant slicing handlebody} of $(K,\tau,h)$ is a handlebody $H\subset S^4$ such that $\overline{\tau}(H)=H$ and $H$ intersects the standard $S^3$ containing $K$ transversely with $h\subset H$ and $\partial H \cap S^3 = K$. We call $\partial H$ an \textit{equivariant slicing surface} and say that $(K,\tau,h)$ \textit{divides} $\partial H$.
\end{definition}

\begin{remark} \label{rem: orientation}
    Since $H\cap S^3$ contains an arc $h$ of the fixpoint set, $\overline{\tau}$ preserves the tangent vector to $h$ at a point of $h$. Since it setwise fixes the hemispheres of $S^4$, it also fixes the normal vector to $S^3$ at that point. Furthermore, the fixed point set is two-dimensional, so we conclude that $\overline{\tau}|_H$ is orientation reversing.
\end{remark}

\begin{remark}Given the equivariant connect sum of two directed strongly invertible knots, as defined in the previous section, we get a natural equivariant connect sum of equivariant slicing surfaces which we will make use of in future constructions.     
\end{remark}

To see that an equivariant slicing surface exists, we take an equivariant Seifert surface $F\subset S^3$ for $(K,\tau,h)$, guaranteed to exist by \cite{boyle2022equivariant}, and consider $F\times [-1,1]\subset~S^4$, as in the non-equivariant setting described in \cite{livingston2015doubly}. $F\times [-1,1]$ is then an equivariant slicing handlebody for $(K,\tau,h)$. We say a surface is \textit{equivariantly unknotted} if it bounds an equivariant handlebody, but do not require the possibly stronger condition that it is equivariant isotopic to some standard embedding. With this in mind, we define the \textit{equivariant double-slice genus}:

\begin{definition}
    The \textit{equivariant double-slice genus} of a directed strongly invertible knot $(K,\tau,h)$, which we denote $\tilde{g}_{ds}(K,\tau,h)$, is the minimal genus of an equivariant slicing handlebody for $(K,\tau,h)$. If $\tilde{g}_{ds}(K,\tau,h)=0$, we say $(K,\tau,h)$ is \textit{equivariantly double-slice}.
\end{definition}

It is clear that the equivariant 4-genus of a directed strongly invertible knot is equal to that of its antipode, as any symmetric surface one bounds is also bounded by the other. It is less clear, and currently unknown to the author, if the same is true for the equivariant double slice genus.

\begin{question}\label{question: antipode}
    Given a directed strongly invertible knot and its antipode, does $\tilde{g}_{ds}(K,\tau,h)=\tilde{g}_{ds}(a(K,\tau,h))$?
\end{question} 

Because this is unknown, we define the \textit{equivariant double-slice genus} of a non-directed strongly invertible knot $(K,\tau)$ to be the minimum of $\tilde{g}_{ds}(K,\tau,h)$ and $\tilde{g}_{ds}(a(K,\tau,h))$. 

The basic lower bounds for the equivariant double-slice genus, $\tilde{g}_{ds}(K,\tau)\geq g_{ds}(K)$ and $\tilde{g}_{ds}(K,\tau)\geq  2\tilde{g}_4(K,\tau)$, do not allow us to differentiate the equivariant double-slice genus from the double-slice genus and the equivariant $4$-genus at the same time, as can be done with Theorem \ref{thm: Main}. 

\subsection{Involutions of handlebodies} \label{sec: handlebodies}

Here we provide a brief discussion of the involutions of handlebodies and, more specifically, equivariant slicing handlebodies. We start by recalling a general result about involutions on handlebodies coming from the work of Kalliongis and McCullough in \cite{kalliongis1996orientation}. In their work, Kalliongis and McCullough discuss multiple decompositions of actions on handlebodies into simpler actions on smaller parts. We will make use of the first decomposition they discuss, called the vertical-horizontal decomposition.
They define an involution on a bundle $\Sigma \times I$ to be \textit{vertical} if it is of the form $1_\Sigma \times r$ and \textit{horizontal} if it is of the form $\sigma \times I$ for $\sigma$ some involution of $\Sigma$, see Figure \ref{fig: decomp}. 

\begin{figure}[h]
\begin{tabular}{p{.48\textwidth}p{.48\textwidth}}
    \begin{center}\includegraphics[width=.9\linewidth]{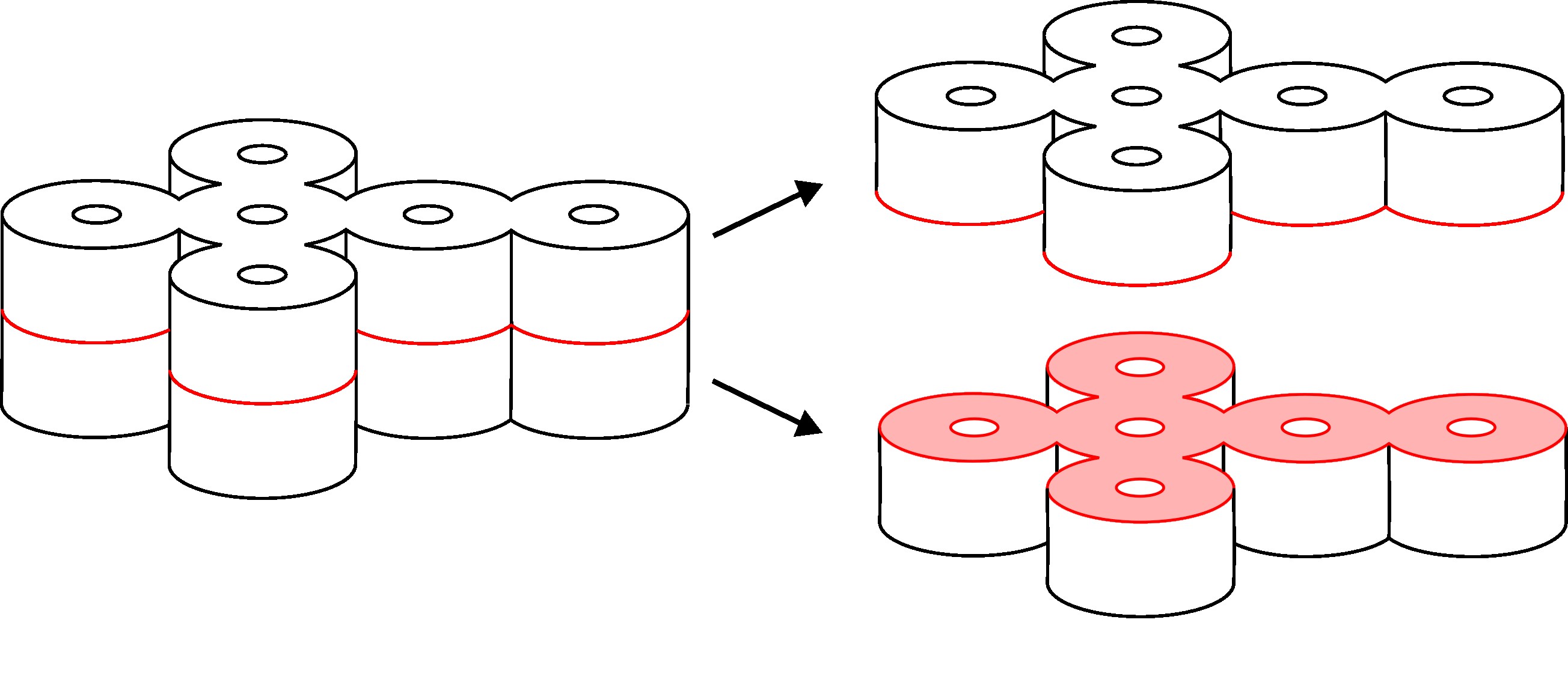}\end{center} &  \begin{center}\includegraphics[width=.9\linewidth]{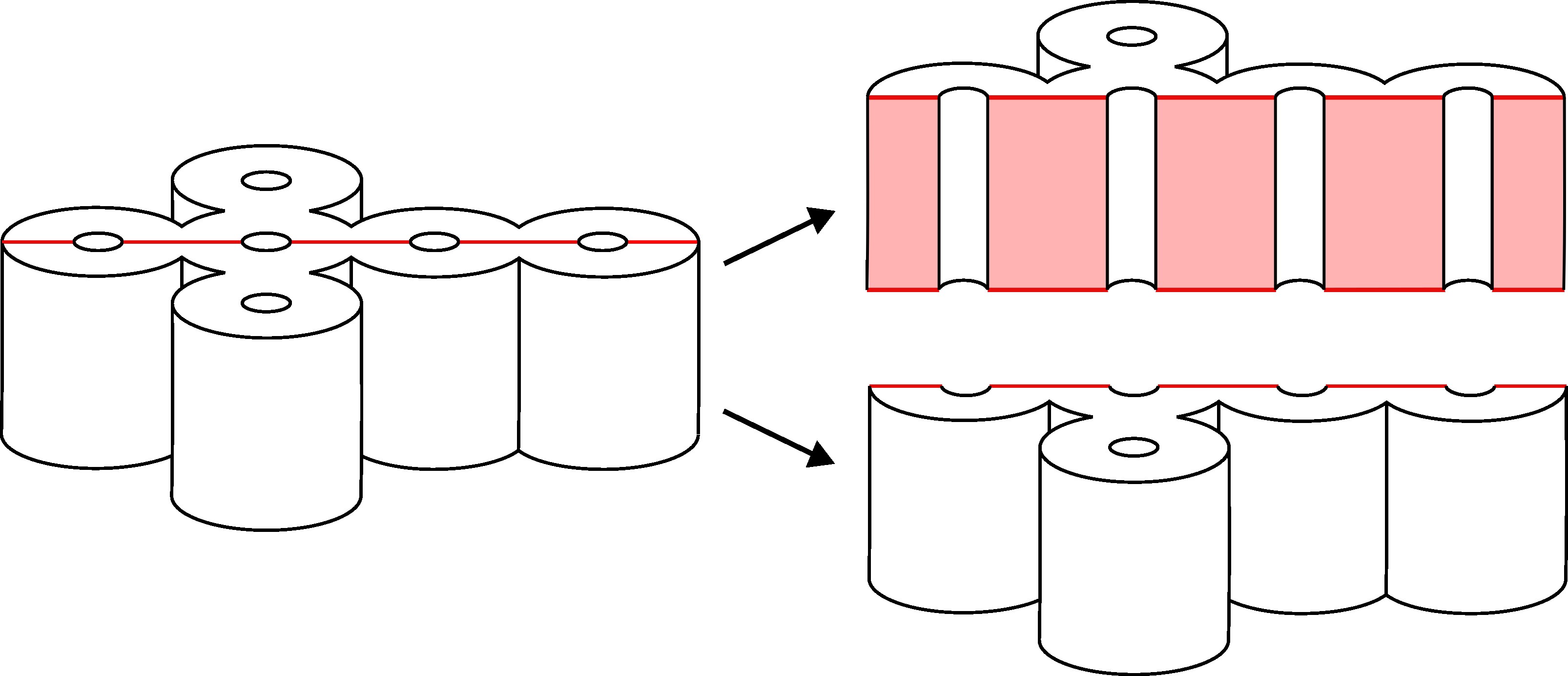} \end{center}
    \end{tabular}
\caption{Example of vertical $H_i$ components (left) and horizontal $H_0$ components (right) with fixed point set in red.}
\label{fig: decomp}
\end{figure}

\begin{theorem}[Theorem 5.1 in \cite{kalliongis1996orientation}] \label{thm: Kall} Let $\tau$ be an orientation reversing involution of a handlebody $H$, and suppose that some component of fix$(h)$ is 2-dimensional. Then $H$ has a decomposition $H = H_0 \cup (\bigcup_{j=1}^r H_j)$, where each piece is $h$-invariant, such that

\begin{enumerate}
    \item $H_0$ is an $I$-bundle over a connected 2-manifold, and the restriction of $h$ to $H_0$ is horizontal. This action may be chosen to be a product action if and only if no component of fix$(h)$ is a point or a M\"obius band. 
    \item Each $H_j$ is an $I$-bundle over a surface of negative Euler characteristic, which is a deformation retract of a component of fix($h)$, and the restriction of $h$ to $H_j$ is vertical.
    \item $\{H_1,\dots, H_r\}$ are pairwise disjoint, and for $1\leq i \leq r$ each $H_0\cap H_i$ is a single 2-disk.
\end{enumerate}
    
\end{theorem}

In the case of an equivariant slicing handlebody, we get added restrictions on the vertical-horizontal decomposition coming from our knowledge of $\overline{\tau}$. 

\begin{lemma} \label{prop: fix}The fixed point set of an equivariant slicing handlebody for a directed strongly invertible knot $(K,\tau,h)$ is the disjoint union of some number of planar surfaces.
\end{lemma}

\begin{proof} Let $H$ be an equivariant slicing handle body for $(K,\tau,h)$. By Remark \ref{rem: orientation}, $\tau$ must be orientation reversing on $H$ with 2-dimensional fixed point set. So, by \cite{kalliongis1996orientation}, the fixed point set of $\tau|_H$ is a collection of surfaces. However, the fixed point set of $\tau|_H$ is a subset of the total fixed point set, a sphere. This means the surfaces comprising the fixed point set are limited to planar surfaces. 

\end{proof}

With this restriction on the fixed point set, we can now rewrite Theorem \ref{thm: Kall} for equivariant slicing handlebodies to give us better restrictions on the $H_i$s.

\begin{proposition} \label{prop: decomp} Let $H$ be an equivariant slicing handlebody for a directed strongly invertible knot $(K,\tau,h)$. Then $H$ has a decomposition $H = H_0 \cup (\bigcup_{j=1}^r H_j)$, where each piece is $\tau$-invariant, such that

\begin{enumerate}
    \item $H_0$ is an $I$-bundle over a connected planar surface, and the restriction of $\tau$ to $H_0$ is horizontal and a product action. 
    \item Each $H_j$ is an $I$-bundle over a planar surface which is not a disk or annulus, which is a deformation retract of a component of fix($\tau)$, and the restriction of $\tau$ to $H_j$ is vertical. 
    \item $\{H_1,\dots, H_r\}$ are pairwise disjoint, and for $1\leq i \leq r$ each $H_0\cap H_i$ is a single 2-disk.
\end{enumerate}
    
\end{proposition}

\begin{proof}
    We begin by verifying that the assumptions of Theorem \ref{thm: Kall} hold for an arbitrary equivariant slicing handlebody $H$, i.e. that $\tau|_H$ is orientation reversing and that fix$(\tau|_H)$ contains a 2-dimensional component. The fact that $\tau|_H$ is orientation reversing was discussed in Remark \ref{rem: orientation}. The fact that fix$(\tau|_H)$ contains a 2-dimensional component comes from the fact that it contains a half-axis of fix$(\tau)\subset S^3$. Since this subset is 1-dimensional, and since $\tau|_H$ is orientation reversing, it must then be a subset of a 2-dimensional subset of fix$(\tau|_H)$. Thus, we know that $H$ has a vertical horizontal decomposition as in Theorem \ref{thm: Kall}.

    We now look at the changes to $(1)$, namely that $H_0$ is an $I$-bundle over a planar surface and that the restriction of $h$ to $H_0$ is horizontal and a product action. The fact that $H_0$ is an $I$-bundle over a planar surface, as opposed to an arbitrary 2-manifold as in Theorem \ref{thm: Kall}, is a direct result of Lemma~\ref{prop: fix}. Similarly, by Lemma~\ref{prop: fix} no component of $fix(\tau)$ is a point or Mobius band. So by (1) of Theorem \ref{thm: Kall}, we have that the restriction of $\tau$ to $H_0$ is horizontal and is a product action. 

    For (2) the only change made from Theorem~\ref{thm: Kall} is noting that the only surfaces that can appear are, by Lemma~\ref{prop: fix}, planar surfaces. By (2) of Theorem \ref{thm: Kall}, each must have a negative Euler characteristic, hence is not a disk or annulus. 

    As there were no changes to (3), this completes the proof.
    
\end{proof}

If instead of orientation preserving extensions we consider orientation reversing extensions, the problem changes dramatically. In this case, the action switches the hemispheres of $S^4$ and would be orientation preserving on the equivariant slicing handlebody for $(K,\tau)$. The main results of this paper, discussed in the next section, do not apply in this setting, as the proofs rely heavily on the behavior of the action on the handlebodies we have just discussed.

\section{Bounds on equivariant double-slice genus} \label{sec: Bounds}

In this section, we prove Theorem \ref{thm: Main} and Theorem \ref{thm: Construction} and discuss relevant examples. We first prove a version of Theorem~\ref{thm: Main} for directed strongly invertible knots, from which Theorem~\ref{thm: Main} follows immediately.

\begin{theorem} \label{thm: ds bound} Let $(K,\tau,h)$ be a directed strongly invertible knot and $K_0$ be the union of $h$ with an arc of $K$ ending on the two fixed points. Then $g_{ds}(K_0)\leq \tilde{g}_{ds}(K,\tau,h)$. 
\end{theorem}

\begin{proof} Let $H$ be a minimal equivariant slicing handlebody for $(K,\tau,h)$. We will start by creating a decomposition of $H$ which allows us to construct a useful slicing handlebody for $K_0$.

We start by decomposing $H$ into $H=H_0 \cup (\bigcup_{j=1}^r H_j)$ as described in Proposition \ref{prop: decomp}. Using this decomposition, we will show that the fixed point set $\overline{\tau}|_H$, which we will abuse notation and refer to as simply $\tau$, separates $H$ into two identical components, $H^1$ and $H^2=\tau(H^1)$. Smith proved that the fixed point set of an involution on a sphere
is a $\mathbb{Z}_2-\check{C}ech$ homology sphere \cite{smith1945transformations,smith1938transformations,smith1939transformations,smith1941transformations,}. This means an involution on a circle has zero or two fixed points. It follows by extending the involution on a planar surface $\Sigma$ to an involution on $S^2$
that a 1-dimensional fixed point set of an involution on $\Sigma$ is either $S^1$ or a collection of properly embedded arcs, either of which separates the surface into identical components.

Thus, since $\tau$ acts trivially on fibers of $H_0$, the fixed point set of $\tau$ separates $H_0$ into two identical components, $H_0^1$ and $H_0^2$, satisfying $g(H_0)\geq g(H_0^i)$ as shown in Figure \ref{fig: decomp}. For the $H_j$, the fixed point set is a planar surface separating the$H_j$ into identical components, $H_j^1$ and $H_j^2$, with $H_j^1=\tau(H_j^2)$ and $g(H_j)=g(H_j^i)$ for $i=1,2$ as shown in Figure \ref{fig: decomp}. Note that all the attaching regions in both $H_0$ and the $H_j$ are disks containing an arc of the fixed point set which are identified. Since both fixed point sets separate, the fixed point set of $H_0\cup H_j$ separates. Iterating, we get that the fixed point set of $H$ separates it into two handle bodies, $H^1$ and $H^2$, with decompositions given by connect sums of the $H_j^1$ and $H_j^2$, respectively. Thus, we have:

\[
g(H)=g(H_0)+\sum_{j=1}^r g(H_j) \geq g(H_0^i) +\sum_{j=1}^r g(H_j^i) = g(H^i).
\]

By construction, the $H^i$ are splitting handlebodies for the two ``halves" of $K$ corresponding to the two different arcs. 
\end{proof}

This allows us to partially reduce questions about the equivariant double-slice genus of a knot $K$ to the non-equivariant double-slice genus of a different knot $K_0$. With this, we are ready to prove Theorem \ref{thm: Construction}.

\begin{figure}
    \begin{tabular}{p{.48\textwidth}p{.48\textwidth}}
        \begin{center}\includegraphics[width=.9\linewidth]{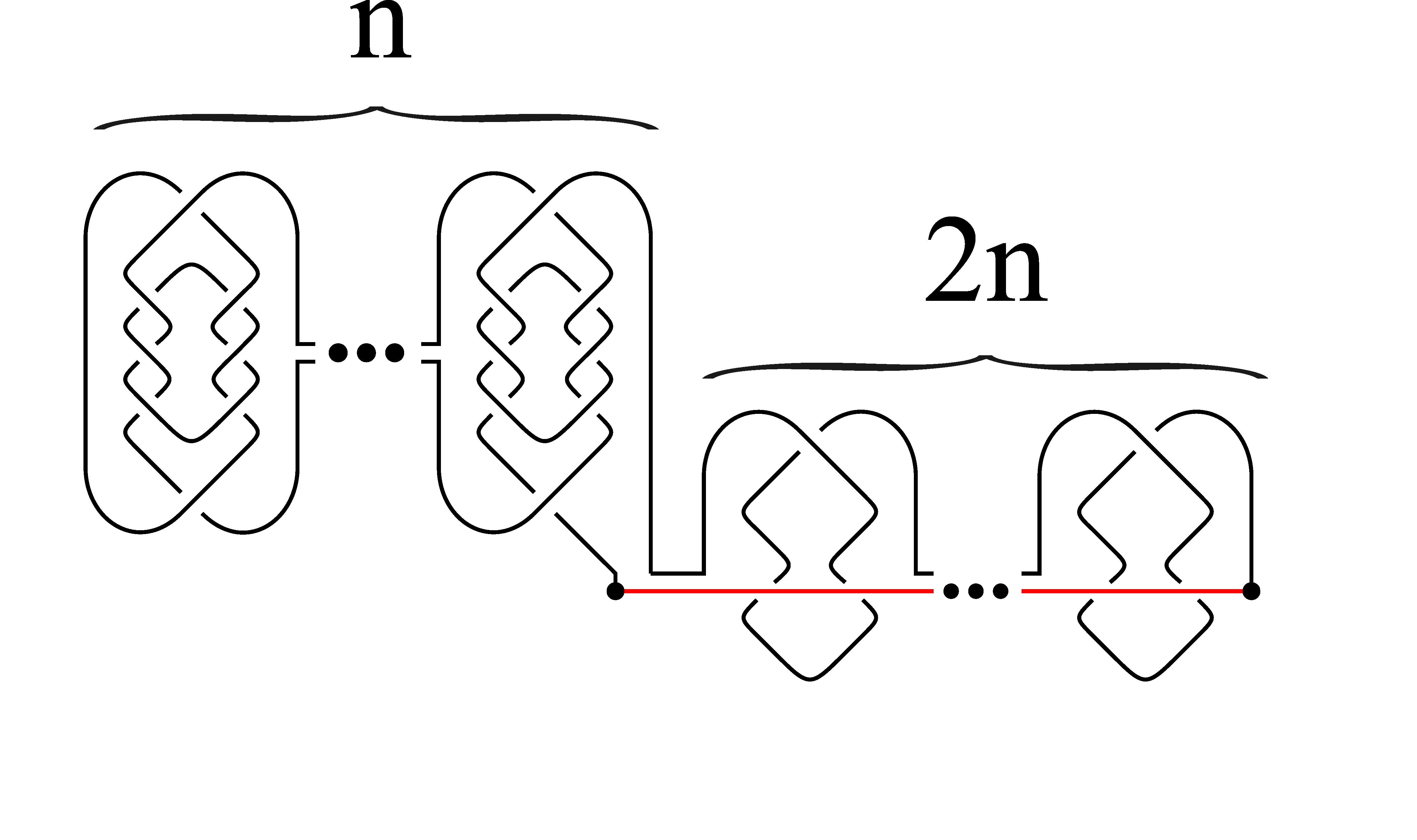}\end{center} &  \begin{center}\includegraphics[width=.9\linewidth]{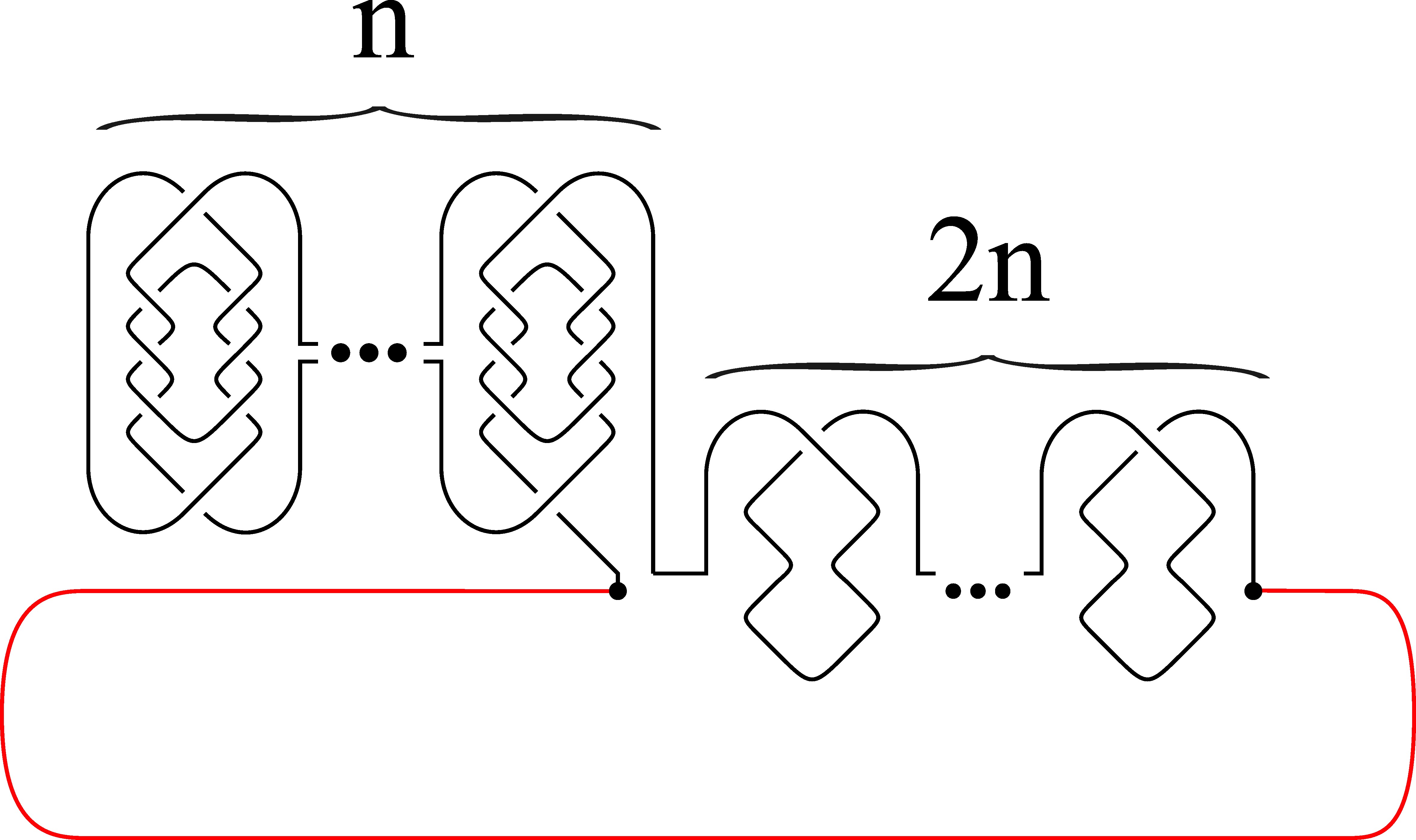} \end{center}
    \end{tabular}
    \caption{The decomposition of $(K_n,\tau)$ with both half-axis.}
    \label{fig: Main Decomp}
    \end{figure}

\begin{proof}[Proof of Theorem \ref{thm: Construction}]
    We will show the knot $K_n$ pictured in Figure \ref{fig:Main Construction} satisfies the desired properties. To see that $K_n$ is double-slice, we note that $K_n = (8_{20}\# -8_{20})^{2n}$. Since $8_{20}\# -8_{20}$ is double-slice, so is the $2n$ connect sum of it with itself, i.e. $K_n$. Now we show that it is equivariantly slice. First, note that $8_{20}$ is the pretzel knot $(3,-3,2)$ and so, by Sakuma \cite{sakuma1986strongly}, is equivariantly slice. Additionally note that since $8_{20}$ is slice, $D(8_{20})$ is equivariantly slice. Since $K_n$ can be seen as an equivariant connect sum of 2$n$ copies of $8_{20}$ and $n$ copies of $D(8_{20})$, which we just said are equivariantly slice, we know that $K_n$ is also equivariantly slice. 

    Lastly, we show that $\tilde{g}_{ds}(K_n,\tau,h_n)\geq n$. To do this, we first see in Figure \ref{fig: Main Decomp} that, regardless of which half-axis we pick, the knot we get from taking an arc of $K_n$ union the half-axis is $\#^n 8_{20}$. Thus, by Theorem \ref{thm: Main}, we have that, for either choice of half-axis, $ \tilde{g}_{ds}(K_n,\tau,h_n) \geq g_{ds}(\#^n 8_{20})$. As discussed in Example \ref{example: 8_20}, using signature invariants and the work of Orson-Powell \cite{orson2020lower}, we can see $g_{ds}(\#^n 8_{20})=n$. Thus, we get our last property, that $\tilde{g}_{ds}(K_n,\tau)\geq n$.
   
\end{proof}

\begin{figure}[h]
    \centering
    \includegraphics[width = .35\linewidth]{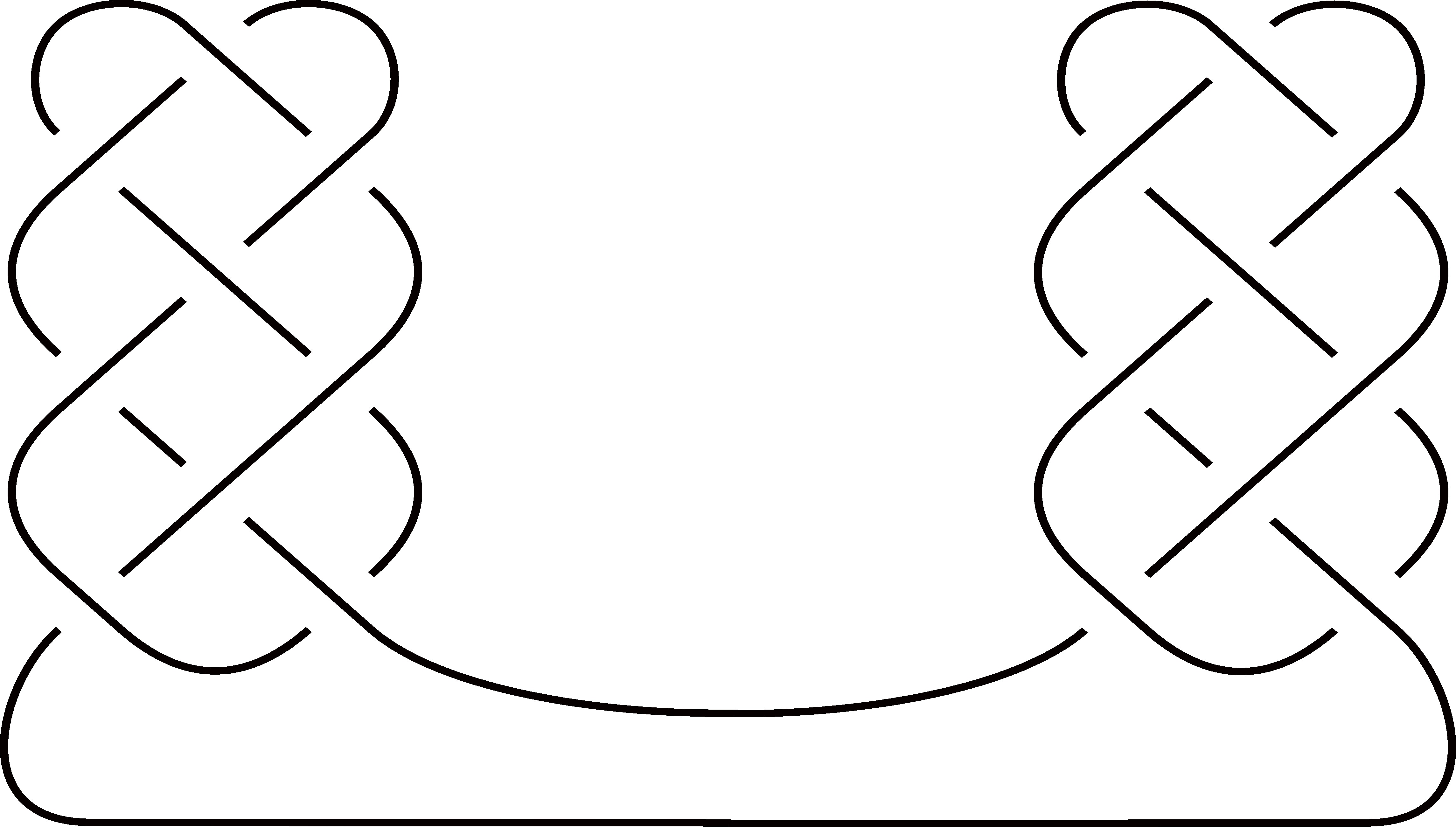}
    \caption{The knot $8_9\#8_9$.}
    \label{fig:8_9 Sum}
\end{figure}

\begin{example}
    For a slightly different construction, we will show that the hyperbolic knot $(S,\tau,h)$ depicted in Figure \ref{fig:8_9 Sum Hyperbolic} is double-slice, equivariantly slice, but not equivariantly double-slice. We will do this by showing the knot $(K,\tau,h)$ depicted in Figure \ref{fig:8_9 Sum} is double-slice and equivariantly slice then identifying an invertible concordance from $(S,\tau,h)$ to $(K,\tau,h)$. Because the constructed concordance will be symmetric, this will show that $(S,\tau,h)$ is also double-slice and equivariantly slice. We will then apply Theorem \ref{thm: Main} to see that it is not equivariantly double-slice.
    
    Since $K$ is the connect sum of the knot $8_{20}$ with itself, and $8_{20}$ is fully amphechiral (meaning $8_{20}=-8_{20}$), we have that $K=8_{20}\#8_{20}$ and is therefore double-slice. Additionally, $8_{20}$is slice,and therefore by Proposition \ref{prop: double}, $K=D(8_{20})$ is equivariantly slice.

    By adding symmetric grabbers to $8_9\#8_9$ to construct $S$ as in Figure \ref{fig:8_9 Sum Hyperbolic}, one gets a symmetric invertible concordance from $(S,\tau,h)$ to $(K,\tau,h)$. Therefore, $(S,\tau,h)$ is double-slice and equivariantly slice.

     Lastly, to see that $(S,\tau,h)$ is not equivariantly double-slice, we apply Theorem \ref{thm: ds bound}. The knot we get by adding a half axis and deleting an arc is $8_{20}$, which has $g_{ds}(8_{20})=1$. Therefore, $\tilde{g}_{ds}(S,\tau,h)\geq 1$, meaning $(S,\tau,h)$ is not equivariantly double-slice. 
\end{example}

\begin{figure}[h]
    \centering
    \includegraphics[width = .65\linewidth]{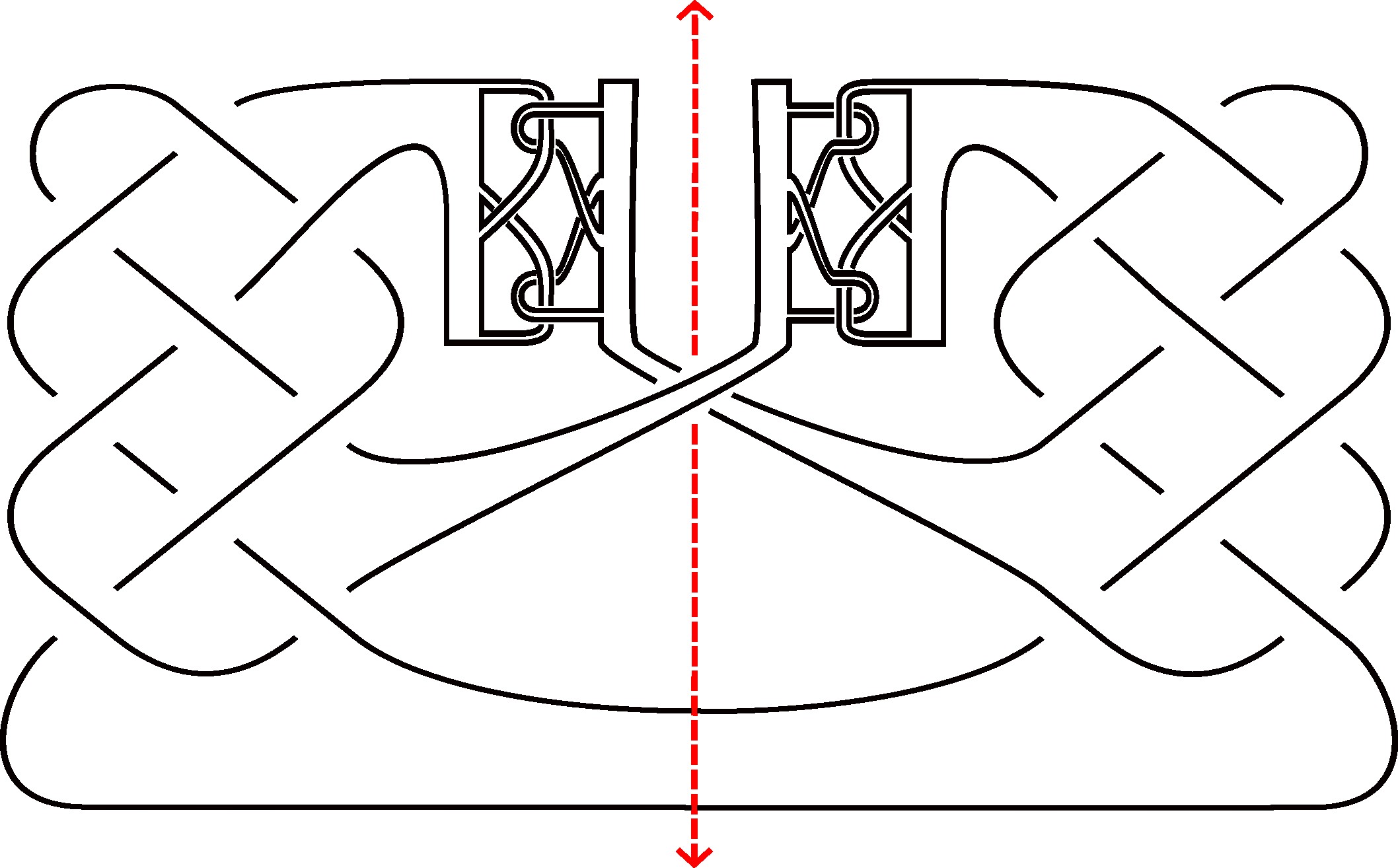}
    \caption{A knot $S$ invertibly concordant to $8_9\# 8_9$.}
    \label{fig:8_9 Sum Hyperbolic}
\end{figure}

\section{Equivariant super-slice genus and equivariantly knotted 2-spheres}\label{sec: Super}

In this section, we define a notion of equivariant super-slice genus and construct bounds similar to those for equivariant double-slice genus in Theorem \ref{thm: Main}. We will use this result in Section \ref{sec: isotopy} to provide examples of unknotted symmetric 2-spheres which are symmetrically knotted.

\subsection{Equivariant super-slice genus}

Similar to the double-slice genus, there is a concept of \textit{super-slice genus} first defined by Wenzhao Chen in \cite{chen2021lower}. The \textit{super-slice genus} of a knot $K$, denoted $g_{ss}(K)$, is the minimal genus of a slicing surface $\Sigma$ for $K$ such that $\Sigma$ is symmetric about $K$. That is to say, $\Sigma\subset S^4$ is the double of a surface $F\subset B^4$ along its boundary $K\subset S^3$. We call $\Sigma$ a \textit{superslicing surface} for $K$, and the handlebody it bounds a \textit{superslicing handlebody}. In the initial conception of super-slice genus by Wenzhao Chen in \cite{chen2021lower}, $g_{ss}(K)$ is the genus of $F$ as opposed to $\Sigma$, meaning it is exactly half of our conception. This change is made to keep the numbers consistent with those of the double-slice genus. If $g_{ss}(K)=0$, we say $K$ is \textit{super-slice}.

While the double-slice genus is a clear lower bound for the super-slice genus of a knot, in our work we will make use of a stronger lower bound proved by Chen in \cite{chen2021lower}:

\begin{theorem}[Chen]\label{thm: Chen}
    Given a knot $K$ let $\Sigma$ be the two-fold branched cover of $S^3$ along $K$. Let $n$ be the minimum number of generators for $H_1(\Sigma;\Z)$. Then $n\leq g_{ss}(K)$.
\end{theorem}

In the same way we were able to define equivariant slicing handlebodies and surfaces, we now define \textit{equivariant superslicing handlebodies} and \textit{equivariant superslicing surfaces}.

\begin{definition}
    Given a strongly invertible knot $(K,\tau,h)$, an \textit{equivariant superslicing handlebody} $H$ of $(K,\tau,h)$ is a superslicing handlebody which is also an equivariant slicing handlebody. We call $\partial H$ an \textit{equivariant superslicing surface} for $(K,\tau,h)$.
\end{definition}

The construction used in Section \ref{sec: EDSG} to guarantee the existence of an equivariant slicing handlebody also guarantees the existence of an equivariant superslicing handlebody. Thus, we can define the \textit{equivariant super-slice genus} of a directed strongly invertible knot $(K,\tau,h)$, denoted $\tilde{g}_{ss}(K,\tau,h)$.

\begin{definition}
    The \textit{equivariant super-slice genus} of a strongly invertible knot $(K,\tau,h)$ is the minimal genus of an equivariant superslicing surface for $(K,\tau,h)$. If  $\tilde{g}_{ss}(K,\tau,h)=0$, we say $(K,\tau)$ is \textit{equivariantly super-slice}.
\end{definition}

As in the equivariant double-slice case, when considering a non-directed strongly invertible knot $(K,\tau)$, we define the \textit{equivariant super-slice genus} of $(K,\tau)$ to be the minimum of the equivariant super-slice genus of the two directed strongly invertible knots obtainable from $(K,\tau)$.

Our primary bound for $\tilde{g}_{ss}(K,\tau,h)$ follows exactly as in the equivariantly double-slice case:

\begin{theorem} \label{thm: super-slice}Let $(K,\tau,h)$ be a directed strongly invertible knot and $K_0$ be an arc of $K$ union $h$. Then $g_{ss}(K_0)\leq \tilde{g}_{ss}(K,\tau,h)$. 
\end{theorem}

The proof of this is exactly the proof of Theorem \ref{thm: Main}, taking the equivariant slicing handlebody to be an equivariant superslicing handlebody. 

\begin{example}
    Using Theorem \ref{thm: super-slice}, one can check that the super-slice knot $J$ in Figure \ref{fig:Super Block} is not equivariantly super-slice with the blue half-axis. When considered with the red axis, it is possible that $a(J,\tau,h)$ is equivariantly super-slice, making it a possible example answering Question \ref{question: antipode} in the negative. 
\end{example}

\begin{figure}[h]
    \centering
    \includegraphics[width=.34\linewidth, angle = -90]{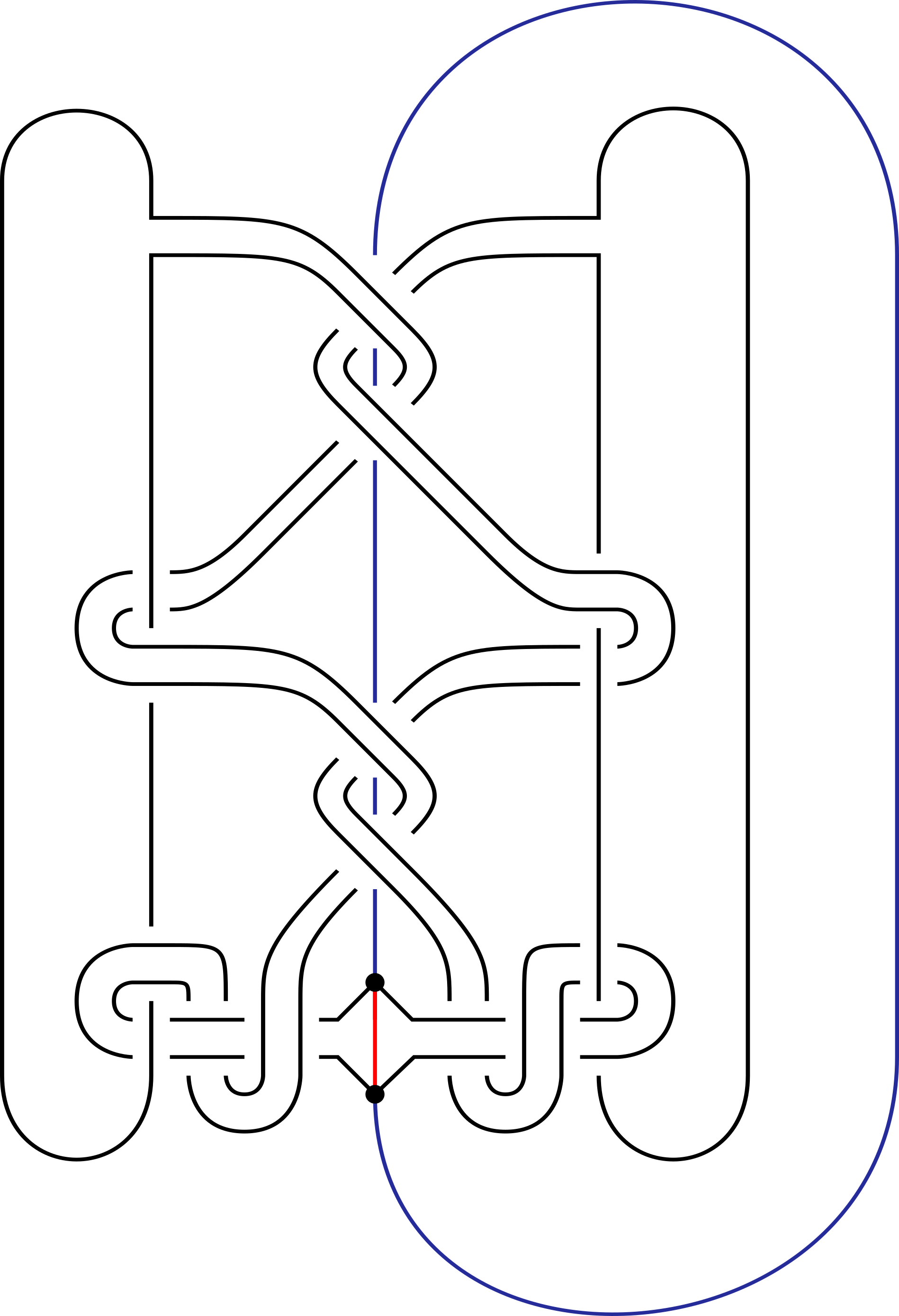}
    \caption{Super-slice knot $(J,\tau,h)$ (blue) and $a(J,\tau,h)$ (red).}
    \label{fig:Super Block}
\end{figure}

\subsection{Symmetrically Knotted 2-spheres} \label{sec: isotopy} For strongly invertible knots, Marumoto \cite{yoshihiko1977relations} proves that any two strongly invertible knots which are isotopic to the unknot are equivariantly isotopic to each other. We will show that this does not extend to equivariantly unknotted 2-spheres in $(S^4,\tau)$. We create symmetric 2-spheres that are unknotted in the classical sense (bound a 3-ball) but do not bound any symmetric 3-ball, contrasting with the properties of strongly invertible knots in $S^3$.


\begin{proof}[Proof of Theorem \ref{thm: 2-sphere}]
    Consider the knot $(J_1,\tau)=(J,\tau,h)\# a(J,\tau,h)$ depicted in Figure \ref{fig:Weird Sphere}. We will construct a slicing 2-sphere, $S^2$, for $K$ such that $\tau(S^2)=S^2$ is not an equivariant slicing sphere, i.e. $S^2$ will be a symmetric 2-sphere bounding a 3-ball which bounds no invariant 3-ball. 
    
    First, we show that $(J_1,\tau)$ is not equivariantly double-slice (with either half-axis), which means any equivariant 2-sphere it appears as a cross section of cannot bound an equivariant 3-ball. This follows immediately from Theorem \ref{thm: ds bound} as the decomposition gives a connect sum of $6_1$ with itself, which is not double-slice. 
    
    Thus, to complete the proof, all we need is a symmetric 2-sphere for $(J_1,\tau)$ which we know is unknotted. To construct this, we first construct such a sphere for $(J,\tau)$. This sphere is the ribbon 2-sphere constructed by taking the double of the obvious ribbon disk for $J$. Since the ribbon disk is symmetric, so is the ribbon 2-sphere. The fact that this 2-sphere is unknotted comes from elementary moves for ribbon surfaces, i.e. a homotopy of core of the band in Figure \ref{fig:Super Block} to a trivial band between the two circles corresponds to an isotopy of the corresponding ribbon 2-sphere. Thus, $(J,\tau)$ appears as the cross-section of a symmetric unknotted 2-sphere which we call $S$.
    
    Note that $a(J_1,\tau,h)$ also appears as a cross section of the same 2-sphere $S$, as this 2-sphere does not depend on the choice of half-axis. Therefore, taking the equivariant connect sum of these two spheres, we get another symmetric sphere, $S^2$, who's intersection with $S^3$ is $(J_1,\tau,h)\# a(J_1,\tau,h)$. Since this $S^2$ is the connect sum of two unknotted spheres, it itself is unknotted. Thus, this $S^2$ is a symmetric sphere bounding a 3-ball which does not bound any invariant 3-ball. 
    
\end{proof}

\begin{figure}
    \centering
    \includegraphics[width=.65 \linewidth]{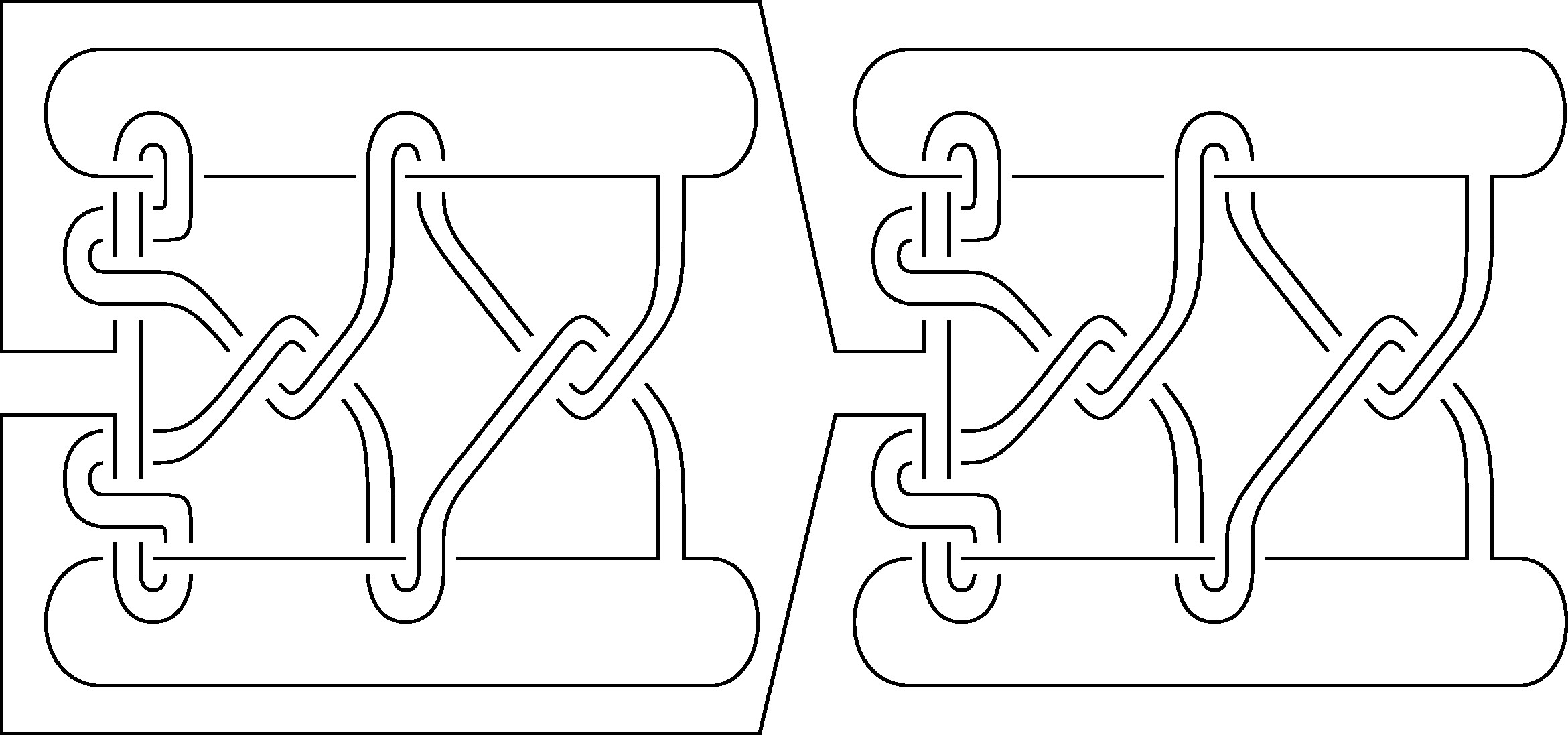}
    \caption{The knot $(J_1,\tau,h)\# a(J_1,\tau,h)$.}
    \label{fig:Weird Sphere}
    \end{figure}

\section{Internal Stabilization Rel Boundary}

In Section \ref{sec:internal}, we describe how to use super-slice genus to obstruct isotopy of surfaces rel boundary and bound their 1-handle stabilization distance. In Section \ref{sec: equivariant stab}, we define a notion of equivariant stabilization distance and extend the results of Section \ref{sec:internal} to this new equivariant setting.

\subsection{Internal Stabilization Bounds from Double Slice Genus}\label{sec:internal}
First, we must define some terms. Let $\Sigma_1,\Sigma_2\subset B^4$ be two properly embedded surfaces with common boundary $K$. The \textit{1-handle stabilization distance} $d_1(\Sigma_1,\Sigma_2)$ from $\Sigma_1$ to $\Sigma_2$ is the minimum number of orientation preserving ambient 1-handles $\{h_i\}$ and $\{h_i'\}$ needed so that $\Sigma_1\cup \{h_i\}$ is smoothly isotopic rel boundary to $\Sigma_2 \cup \{h_i'\}$, as defined in \cite{miller2020stabilization}. With this we are ready to prove Theorem \ref{thm: stab}

\begin{proof}{Proof of Theorem \ref{thm: stab}}
    Let $d=d_1(\Sigma_1,\Sigma_2)$ and let $\Sigma_1'=\Sigma_1\cup \{h_i\}_{i=1}^d$ and $ \Sigma_2' = \Sigma_2 \cup \{h_i'\}_{i=1}^d$ be stabilized surfaces that are isotopic rel boundary. Since $\Sigma_1\cup_K -\Sigma_2 \subset (B^4,\Sigma_1)\cup_{(S^3,K)}(B^4,\Sigma_2)$ is unknotted, the stabilized surface $\Sigma' = \Sigma_1'\cup_K -\Sigma_2' \subset (B^4,\Sigma_1')\cup_{(S^3,K)}(B^4,\Sigma_2')$ is also unknotted, as it is an unknotted handlebody $\Sigma_1 \cup_K -\Sigma_2$ union handles. Since $\Sigma_1'$ is isotopic rel boundary to $\Sigma_2'$, we can isotope $\Sigma'$ in $S^4$ rel $B^4$ (the hemisphere containing $\Sigma_1'$) to get the double of $(B^4,\Sigma_1')$. This means that the double of $(B^4,\Sigma_1')$  is a superslicing surface for $K$ and therefore $h+d \geq \frac{g_{ss}(K)}{2}$, i.e. $d\geq \frac{g_{ss}(K)}{2}-h$. Note we have $\frac{g_{ss}(K)}{2}$, not $g_{ss}(K)$, because we want the genus of the surface bounded by the knot, not its double. Thus, $d_1(\Sigma_1,\Sigma_2)\geq \frac{g_{ss}(K)}{2}-h$.
\end{proof}

Letting $K$ be double-slice, we are able to use this to obstruct certain slice disks of $K$ from being isotopic rel boundary, from which Corollary \ref{cor: stab} immediately followl.

In \cite{livingston2015doubly} a survey of double-slice knots with 12 or fewer crossings is conducted and they find 20 knots which are double-slice and provide explicit description of the double slicing via bands. Of these 20 knots, there are 17 with Alexander polynomial not equal to 1, and therefore not super-slice. Thus, for these 17 knots, the band diagrams given in \cite{livingston2015doubly} depict slice discs which are not isotopic rel boundary. This fact has been proven for the first of these 17 knots, $9_{46}$, by others including by Miller and Powell \cite{miller2020stabilization} and Sunderberg and Swann \cite{sundberg2021relative}.

Moreover, of the 17 knots with Alexander polynomial not equal to 1, the following 15 knots have $H_1(\Sigma,\Z)=\Z_n^2$ for some $n$, where $\Sigma$ is the 2-fold branch cover of $S^3$ along one of the knots $K$:

\begin{align*}
    9_{46}, 10_{99}, 10_{123}, 10_{155}, 11n_{74}, 12a_{427}, 12a_{1105}, 12n_{268}, \\ 12n_{397}, 12n_{414}, 12n_{605}, 12n_{636}, 12n_{706}, 12n_{817}, 12n_{838}
\end{align*}

 Thus, if you take the connect sum $\#^m K$, where $K$ is any of the 15 knots, you get that $H_1(\Sigma,\Z)=\Z_n^{2m}$. By Theorem \ref{thm: Chen}, this means that $g_{ss}(\#^m K)\geq 2m$. Combining this fact with Theorem \ref{thm: stab}, we get that any two slice disks arising from the same double slicing of $\#^m K$ have stabilization distance at least $m$. 

\subsection{Equivariant Stabilization}
\label{sec: equivariant stab}

We now define an equivariant notion of 1-handle stabilization distance and reprove Theorem \ref{thm: stab} in this equivariant setting.

\begin{definition}
    Let $\rho$ be a smooth $\Z_p$ action on $B^4$, $K\subset S^3$ be a $\rho$-invariant knot, and $\Sigma_1, \Sigma_2\subset B^4$ be $\tau$-invariant properly embedded surfaces with common boundary $K$. The \textit{equivariant 1-handle stabilization distance}, denoted $\tilde{d}_\rho(\Sigma_1,\Sigma_2)$, is the minimal number of orientation preserving ambient 1-handles $\{h_i\}$ and $\{h_i'\}$ needed so that $\rho(\Sigma_i\cup \{h_i\}) = \Sigma_i\cup \{h_i\} $ for $i\in \{1,2\}$ and $\Sigma_1\cup \{h_i\}$  is smoothly equivariantly isotopic rel boundary to $\Sigma_2 \cup \{h_i'\}$.
\end{definition}

Restricting our attention to the action $\overline{\tau}$ discussed in Section \ref{sec: Equivariant 4-Genus} we can ask about equivariant stabilization distance bounds coming from equivariant double-slice and equivariant super-slice genus. More specifically, we get the statement of Theorem \ref{thm: eq stabilization},
the proof of which follows exactly as in the non-equivariant setting. One interesting question about equivariant stabilization that Theorem \ref{thm: stab} may be useful for is the following:

\begin{question}
    For every $n\in \N$, does there exist $\overline{\tau}$-invariant surfaces $\Sigma_1,\Sigma_2\subset B^4$ which are isotopic but have $\tilde{d}_1^\tau(\Sigma_1,\Sigma_2)=n$? 
\end{question}

\printbibliography

\end{document}